\documentclass[twocolumn,amsthm]{autart}    

\usepackage{graphicx}          
\usepackage{url}
\usepackage{amsmath,mathtools,amssymb}
\usepackage{bm}
\usepackage[inline]{enumitem}
\usepackage{color}
\usepackage[noadjust]{cite}

\DeclareRobustCommand{\svdots}{
  \vbox{%
    \baselineskip=0.33333\normalbaselineskip
    \lineskiplimit=0pt
    \hbox{.}\hbox{.}\hbox{.}%
    \kern-0.2\baselineskip
  }%
}

\newcounter{assumption}
\newcounter{theorem}
\newcounter{definition}
\newcounter{lemma}
\newcounter{corollary}
\newcounter{remark}

\begin{document}

\begin{frontmatter}

\title{On Poly-Quadratic Stabilizability and Detectability of~Polytopic~LPV~Systems\thanksref{footnoteinfo}} 

\thanks[footnoteinfo]{This research received funding from the European Research Council (ERC) under the Advanced ERC grant agreement PROACTHIS, no. 101055384.\\ Corresponding author: T.~J.~{Meijer}}

\author[TUe]{T.~J.~Meijer}\ead{t.j.meijer@tue.nl},      
\author[TUe]{V.~S.~Dolk}\ead{v.s.dolk@tue.nl},    
\author[TUe]{W.~P.~M.~H.~Heemels}\ead{m.heemels@tue.nl} 

\address[TUe]{Department of Mechanical Engineering, Eindhoven University of Technology, Eindhoven, The Netherlands}  


\begin{abstract}                          
	In this technical communique, we generalize the well-known Lyapunov-based stabilizability ($S-ASA^\top+BB^\top\succ 0$) and detectability ($P-A^\top PA+C^\top C\succ 0$) tests for discrete-time linear time-invariant systems to polytopic linear parameter-varying systems using the class of so-called poly-quadratic Lyapunov functions. 
\end{abstract}


\end{frontmatter}

\section{Introduction}
Controllability, observability, stabilizability and detectability properties are cornerstones of modern control theory. Especially in the context of discrete-time (DT) linear time-invariant (LTI) systems of the form
\begin{equation}
        x_{k+1} = Ax_k + Bu_k, \quad y_k = Cx_k,
    \label{eq:lti-system}
\end{equation} 
many classical characterizations
are available for verifying these 
properties~\cite{Hespanha2018}. In particular, the system~\eqref{eq:lti-system} is detectable (in the sense that there exists an observer that asymptotically recovers the state) if and only if there exists a symmetric positive-definite matrix $P$ such that~\cite{Hespanha2018}
\begin{equation}
    P - A^\top PA + C^\top C\succ 0.
    \label{eq:lti-det-lmi}
\end{equation}
Interestingly,~\eqref{eq:lti-det-lmi} does not contain a term that is directly related to an observer gain for which the resulting error system is asymptotically stable (but we can construct such a gain based on a solution $P$ to~\eqref{eq:lti-det-lmi}). Similarly, the system~\eqref{eq:lti-system} is stabilizable (in the sense that there exists a state-feedback controller that asymptotically stabilizes the system) if and only if there exists a symmetric positive-definite matrix $S$ satisfying~\cite{Hespanha2018}
\begin{equation}
    S - ASA^\top +BB^\top \succ 0,
    \label{eq:lti-stab-lmi}
\end{equation}
which, as for detectability, does not contain a term that directly relates to a stabilizing controller gain (but such a gain can be constructed based on a solution $S$ to~\eqref{eq:lti-stab-lmi}).

In this technical communique, we are interested in extensions of these celebrated Lyapunov-based tests for detectability~\eqref{eq:lti-det-lmi} and stabilizability~\eqref{eq:lti-stab-lmi} to DT polytopic linear parameter-varying (LPV) systems, of the form
\begin{equation}
		x_{k+1} = A(p_k)x_k + Bu_k,\quad 
        y_k = Cx_k,	
	\label{eq:system}
\end{equation}
where $x_k\in\mathbb{R}^{n_x}$, $u_k\in\mathbb{R}^{n_u}$ and $y_k\in\mathbb{R}^{n_y}$ denote, respectively, the state, control input and measured output at time $k\in\mathbb{N}$. Here, the unknown parameter $p_k$ belongs to a known set $\mathbb{P}$, i.e., $p_k\in\mathbb{P}\subseteq\mathbb{R}^{n_p}$ for all $k\in\mathbb{N}$. The set of admissible parameter sequences, denoted $\mathcal{P}$, is the set of all parameter sequences $\bm{p}=\{p_k\}_{k\in\mathbb{N}}$ with $p_k\in\mathbb{P}$ for all $k\in\mathbb{N}$.  The matrix-valued function $A\colon\mathbb{P}\rightarrow\mathbb{R}^{n_x\times n_x}$ and matrices $B\in\mathbb{R}^{n_x\times n_u}$ and $C\in\mathbb{R}^{n_y\times n_x}$ are known. We require $B$ and $C$ to be constant, which is a common assumption in this context, see, e.g.,~\cite{Daafouz2001,Heemels2010}.
We focus on polytopic systems~\eqref{eq:system}, whose specific mathematical structure facilitates the development of LMI-based tools, see, e.g.,~\cite{Heemels2010,Daafouz2001,Oliveira2022,Pereira2021,Pandey2017b,Pandey2017,Halimi2013}.
\setcounter{thm}{\theassumption}\stepcounter{assumption}
\begin{assum}\label{asm:polytopic-set}
    The parameter set $\mathbb{P}$ is contained within a polytope, i.e., there exist vertices $v_i\in\mathbb{R}^{n_p}$, and continuous
    functions $\xi_i~\colon~\mathbb{P}\rightarrow\mathbb{R}_{\geqslant 0}$, $i\in\mathcal{N}\coloneqq\mathbb{N}_{[1,N]}$, such that the mapping $\xi\coloneqq (\xi_1,\xi_2,\hdots,\xi_N)^\top$ satisfies $\xi(\mathbb{P})\subseteq\mathbb{X}\coloneqq\{\mu=(\mu_1,\mu_2,\hdots,\mu_N)^\top\in\mathbb{R}^{N}_{\geqslant 0}~|~\sum_{i\in\mathcal{N}}\mu_i=1\}$, and
    \begin{equation*}
        \sum_{i\in\mathcal{N}} \xi_i(\pi)v_i = \pi,\text{ for any }\pi\in\mathbb{P}.
    \end{equation*}
\end{assum}
\setcounter{thm}{\thedefinition}\stepcounter{definition}
\begin{defn}\label{dfn:polytopic}
	The system~\eqref{eq:system} is said to be polytopic, if there exist matrices $A_i\in\mathbb{R}^{n_x\times n_x}$, $i\in\mathcal{N}$, such that
\begin{equation}
    A(\pi) = \sum_{\mathclap{i\in\mathcal{N}}}\xi_i(\pi)A_i,\text{ for any }\pi\in\mathbb{P}.
	\label{eq:poly-A}
\end{equation}
    If, in addition, the functions $\xi_i$, $i\in\mathcal{N}$, satisfy $\{\bm{e}_i\}_{i\in\mathcal{N}}\subset\xi(\mathbb{P})$, where $\bm{e}_i=(0_{i-1}^\top,1,0_{n-i}^\top)^\top\in\mathbb{R}^n$, $i\in\mathbb{N}_{[1,n]}$, is the $i$-th $n$-dimensional elementary basis vector, then the system~\eqref{eq:system} is said to be strictly polytopic.
\end{defn}

Many synthesis LMIs have been developed by exploiting the polytopic structure of both the system itself and a specific class of Lyapunov functions, called poly-quadratic Lyapunov functions (poly-QLF), see, e.g.,~\cite{Heemels2010,Daafouz2001,Oliveira2022,Pereira2021,Pandey2017b,Pandey2017,Halimi2013}, which we formally describe in Definition~\ref{dfn:polyQS} below. These synthesis LMIs implicitly assume properties such as stabilizability~\cite{Daafouz2001,Pereira2021,Pandey2017b} and detectability~\cite{Heemels2010,Oliveira2022,Pandey2017,Halimi2013}. These results would benefit from novel tests to study stabilizability and detectability properties \emph{in an isolated manner} (in the sense of not including the observer/controller gain parameters directly in the conditions, such as in~\eqref{eq:lti-det-lmi} and~\eqref{eq:lti-stab-lmi} in the LTI setting) also for LPV systems. To the best of our knowledge, no generalizations of the Lyapunov-based tests in~\eqref{eq:lti-det-lmi} and~\eqref{eq:lti-stab-lmi} to LPV systems are available in the literature. In this technical communique, we fill this gap by extending the celebrated conditions~\eqref{eq:lti-det-lmi} and~\eqref{eq:lti-stab-lmi} to obtain novel conditions, based on poly-QLFs, for polytopic LPV systems that can be used to characterize stabilizability and detectability in an isolated manner. 

{\bf Notation.} The sets of real, non-negative real and natural numbers are denoted, respectively, $\mathbb{R}=(-\infty,\infty)$, $\mathbb{R}_{\geqslant 0}=[0,\infty)$ and $\mathbb{N}=\{0,1,2,\hdots\}$. We denote $\mathbb{N}_{[n,m]}=\{n,n+1,n+2,\hdots,m\}$ for $n,m\in\mathbb{N}$. The set of $n$-by-$n$ symmetric matrices is denoted by $\mathbb{S}^n$. $M\succ 0$ and $M\succcurlyeq 0$ mean, respectively, that $M\in\mathbb{S}^n$ is positive definite and positive semi-definite. The sets of such matrices are $\mathbb{S}^n_{\succ 0}$ and $\mathbb{S}^{n}_{\succcurlyeq 0}$, respectively. The symbol $I$ is an identity matrix of appropriate dimensions and $\star$ completes a symmetric matrix, e.g., $\begin{bmatrix}\begin{smallmatrix} A & \star\\B & C\end{smallmatrix}\end{bmatrix}=\begin{bmatrix}\begin{smallmatrix}A & B^\top\\B & C\end{smallmatrix}\end{bmatrix}$. For $A\in\mathbb{R}^{n\times m}$, $A^\top$ is its transpose, $\operatorname{He}(A)=A+A^\top$ and $A_{\perp}$ is a matrix whose columns are a basis for $\operatorname{ker}A$. Moreover, $\operatorname{diag}\{A_i\}_{i\in\mathbb{N}_{[1,n]}}$ is a block-diagonal matrix with blocks $A_i$, $i\in\mathbb{N}_{[1,n]}$. The notation $x\leftarrow y$ means ``replace'' $x$ by $y$.

\section{Parameter-sequence-dependent systems}
In this section, we introduce the broader class of so-called \emph{parameter-sequence-dependent} (PSD) systems which generalizes~\eqref{eq:system}, given by
\begin{equation}\label{eq:PSD}
		x_{k+1} = \mathcal{A}(k,\bm{p})x_k + Bu_k,\quad y_k = Cx_k,	
\end{equation}
where the matrix-valued function $\mathcal{A}~\colon~\mathbb{N}\times\mathcal{P}\rightarrow\mathbb{R}^{n_x\times n_x}$ depends not only on $p_k$ but on the entire parameter sequence $\bm{p}$ and on the current time $k$. Observe that~\eqref{eq:PSD} includes (polytopic) LPV systems~\eqref{eq:system} as a special case. At first glance, the dependency on $\bm{p}$ may seem strange but it appears naturally in the analysis presented later on.

\subsection{Stability notion}
Next, we introduce an important stability notion based on poly-QLFs for the class of PSD systems~\eqref{eq:PSD} introduced before. Let $x_k(\bm{p},\bm{u},x_0)$ be the solution to~\eqref{eq:PSD} at time $k\in\mathbb{N}$ for initial state $x_0\in\mathbb{R}^{n_x}$, parameter sequence $\bm{p}\in\mathcal{P}$ and input sequence $\bm{u}=\{u_j\}_{j\in\mathbb{N}}$ with $u_j\in\mathbb{R}^{n_u}$ for all $j\in\mathbb{N}$.  
We consider the notion of global uniform asymptotic stability (GUAS)~\cite{Khalil1996-2nd-edition}, in which uniformity with respect to $\bm{p}$ (and $x_0$) is considered, see~\cite[Definition 3]{Meijer2024-cones-automatica}. This is without loss of generality since, for the class of LPV systems considered here, uniformity with respect to initial time $k_0$, which is used in the general definition of GUAS, is equivalent to uniformity in $\bm{p}\in\mathcal{P}$ since any time-shifted version of $\bm{p}$ is also in $\mathcal{P}$. We also note that, for the class of LPV systems considered here, uniform asymptotic stability (UAS) and GUAS are equivalent, see, e.g.,~\cite{Amato2006}.
For polytopic systems~\eqref{eq:system}, poly-QLFs have been used extensively, see, e.g.,~\cite{Daafouz2001,Heemels2010,Oliveira2022,Pereira2021}, for certifying GUAS because they \begin{enumerate*}[label=(\alph*)] \item can be constructed using LMI-based techniques, and \item are less conservative than traditional common quadratic Lyapunov functions \end{enumerate*}. These poly-QLFs lead us to define poly-quadratic stability (poly-QS) for PSD systems~\eqref{eq:PSD} below.
\setcounter{thm}{\thedefinition}\stepcounter{definition}
\begin{defn}\label{dfn:polyQS}
	The PSD system~\eqref{eq:PSD} is said to be poly-QS, if it admits a poly-QLF, i.e., a function $V\colon\mathbb{P}\times\mathbb{R}^{n_x}\rightarrow\mathbb{R}_{\geqslant 0}$ of the form
	\begin{equation}
		V(\pi,x) = x^\top P(\pi)x\text{ with }P(\pi) = \sum_{i\in\mathcal{N}}\xi_i(\pi)\bar{P}_i,
		\label{eq:poly-lyap}
	\end{equation}
	where $\bar{P}_i\in\mathbb{S}^{n_x}$ and $\xi_i$, $i\in\mathcal{N}$, are the functions in~\eqref{eq:poly-A}, satisfying, for some $a_1,a_2,a_3\in\mathbb{R}_{>0}$, and for all $\bm{p}\in\mathcal{P}$, $k\in\mathbb{N}$, and $x_{k+1},x_k\in\mathbb{R}^{n_x}$ with $x_{k+1}=\mathcal{A}(k,\bm{p})x_k$,
	\begin{align}
		a_1\|x_k\|^2&\leqslant V(p_k,x_{k})\leqslant a_2\|x_k\|^2,\label{eq:lyap-bounds}\\
		V(p_{k+1},x_{k+1})&\leqslant V(p_k,x_k)-a_3\|x_k\|^2.\label{eq:lyap-desc}%
	\end{align}
\end{defn}
\vspace*{-\belowdisplayskip}
\noindent Poly-QS implies GUAS~\cite{Khalil1996-2nd-edition}, however, the converse is in general not true, as shown in~\cite[Section 4.1]{Meijer2024-cones-automatica} by means of a counterexample, i.e., a system of the form~\eqref{eq:system} (and, thereby, of the form~\eqref{eq:PSD}) that is GUAS but not poly-QS. Thus, some degree of conservatism is introduced by considering poly-QS. The use of poly-QLFs in the context of polytopic systems is strongly related to the use of mode-dependent Lyapunov functions for switched linear systems (SLSs)~\cite{Mason2012} and, thus, the results in the sequel are also relevant for SLSs~\cite{Athanasopoulos2014,Philippe2016,Liberzon2003}. Several necessary and sufficient LMI-based conditions, e.g.,~\cite{Daafouz2001,Heemels2010,Pandey2018}, for poly-QS of polytopic systems~\eqref{eq:system} are found in the literature.

\subsection{Detectability notion}\label{sec:detectability}
We are interested in the existence of an observer for~\eqref{eq:system}, which results in the corresponding estimation error system, see~\eqref{eq:error-system} below, being poly-QS and, thereby, GUAS. The observer structure that we consider is given by
\begin{equation}
	\hat{x}_{k+1} = A(p_k)\hat{x}_k + L(k,\bm{p})(C\hat{x}_k-y_k) + Bu_k,
	\label{eq:observer}
\end{equation}
where $\hat{x}_k\in\mathbb{R}^{n_x}$ denotes the state estimate at time $k\in\mathbb{N}$ and $L\colon\mathbb{N}\times\mathcal{P}\rightarrow\mathbb{R}^{n_x\times n_y}$ is the to-be-designed observer gain. The error $e_k\coloneqq \hat{x}_k-x_k$, $k\in\mathbb{N}$, is governed by
\begin{equation}
	e_{k+1} = (A(p_k)+ L(k,\bm{p})C)e_k,
	\label{eq:error-system}
\end{equation}
which remains linear in $e$. This class of observers clearly includes the observers~\eqref{eq:observer} with $L(k,\bm{p}) = \bar{L}(p_k)$, $\bm{p}\in\mathcal{P}$, $k\in\mathbb{N}$, for some $\bar{L}\colon\mathbb{P}\rightarrow\mathbb{R}^{n_x\times n_y}$, considered in, e.g.,~\cite{Heemels2010,Pandey2018,Oliveira2022}, as special cases but it also allows $L$ that depend on past or future values of $p_k$. In addition, unlike, e.g.,~\cite{Heemels2010,Pandey2018}, we do not require that $L$ is polytopic or time-independent. Hence,~\eqref{eq:observer} is quite general and allows us to exploit knowledge of the past or future of $\bm{p}$, if available, to design observers that render~\eqref{eq:error-system} poly-QS. 

\setcounter{thm}{\thedefinition}\stepcounter{definition}
\begin{defn}\label{dfn:poly-quadratic-detectability}
	The polytopic system~\eqref{eq:system} is said to be poly-Q detectable, if the error system~\eqref{eq:error-system} is poly-QS for some function $L\colon\mathbb{N}\times\mathcal{P}\rightarrow\mathbb{R}^{n_x\times n_y}$.
\end{defn}
We deliberately do not restrict the form of the observer gain in~\eqref{eq:observer} a priori, but we aim to let the poly-Q detectability requirement (i.e.,~\eqref{eq:lyap-bounds} and~\eqref{eq:lyap-desc} for~\eqref{eq:error-system}) reveal possible simpler structures for~\eqref{eq:observer}. In fact, we will show that if a general $L$ exists for which~\eqref{eq:error-system} is poly-QS, then also a simpler $L$ depending only on $p_k$ can be used. Starting from this general setting also strengthens a possible necessary detectability condition.

\subsection{Stabilizability notion}
Next, we are interested in the existence of stabilizing state-feedback laws of the form 
\begin{equation}
	u_k = K(k,\bm{p})x_k\text{ for some }K\colon\mathbb{N}\times\mathcal{P}\rightarrow\mathbb{R}^{n_u\times n_x},
	\label{eq:fb-law}
\end{equation}
for~\eqref{eq:system}. These are more general feedback laws than those considered in, e.g.,~\cite{Daafouz2001,Pereira2021}, where $K(k,\bm{p})=\bar{K}(p_k)$, $\bm{p}\in\mathcal{P}$, $k\in\mathbb{N}$, for some polytopic function $\bar{K}\colon\mathbb{P}\rightarrow\mathbb{R}^{n_u\times n_x}$, in the sense that $K$ is not necessarily polytopic and may depend on the entire parameter sequence $\bm{p}$. Similar comments apply as for the detectability case above.

\setcounter{thm}{\thedefinition}\stepcounter{definition}
\begin{defn}\label{dfn:poly-q-stabilizability}
	The polytopic system~\eqref{eq:system} is said to be poly-Q stabilizable, if the closed-loop system
     \begin{equation}
    	x_{k+1} = (A(p_k) + BK(k,\bm{p}))x_k
    	\label{eq:closed-loop}
    \end{equation}
    is poly-QS for some function $K\colon\mathbb{N}\times\mathcal{P}\rightarrow\mathbb{R}^{n_u\times n_x}$.
\end{defn}

\section{Main results}

\subsection{Poly-Q-detectability test}
Next, we propose novel LMI-based conditions that generalize the Lyapunov-based LTI detectability test~\eqref{eq:lti-det-lmi}.
\setcounter{thm}{\thetheorem}\stepcounter{theorem}
\begin{thm}\label{thm:poly-q-detectability-analysis}
	The strictly polytopic system~\eqref{eq:system} is poly-Q detectable, if and only if there exist matrices $\bar{P}_i\in\mathbb{S}^{n_x}_{\succ 0}$, $i\in\mathcal{N}$, such that
	\begin{equation}
		\bar{P}_i-A_i^\top \bar{P}_jA_i+C^\top C\succ 0,\text{ for all }i,j\in\mathcal{N}.
		\label{eq:poly-q-detectability-analysis}
	\end{equation}
	Moreover, in that case,~\eqref{eq:observer} with, for $k\in\mathbb{N}$ and $\bm{p}\in\mathcal{P}$,
	\begin{equation}
		L(k,\bm{p}) = -\sum_{i\in\mathcal{N}}\xi_i(p_k) A_i(\bar{P}_i+C^\top C)^{-1}C^\top
		\label{eq:poly-q-detectability-analysis-gain}
	\end{equation}	 
	renders the error system~\eqref{eq:error-system} poly-QS with poly-QLF~\eqref{eq:poly-lyap}. If the system~\eqref{eq:system} is polytopic but not strictly polytopic, the conditions~\eqref{eq:poly-q-detectability-analysis} are sufficient. 
\end{thm}
\vspace*{-\belowdisplayskip}
\begin{proof}
    \underline{Sufficiency:} Suppose that there exist $\bar{P}_i\in\mathbb{S}_{\succ 0}$ and $\bar{S}_i\coloneqq\bar{P}_i^{-1}$, $i\in\mathcal{N}$, for which~\eqref{eq:poly-q-detectability-analysis} holds. By applying Schur complement twice, we obtain
 \begin{equation}\label{eq:twice-schur1}
    \bar{S}_j - A_i(\bar{P}_i+C^\top C)^{-1}A_i^\top \succ 0,\text{ for all }i,j\in\mathcal{N}.
 \end{equation}
 Let $\widehat{\mathcal{P}}_i\coloneqq (\bar{P}_i+C^\top C)^{-1}$ and $L_i \coloneqq -A_i\widehat{\mathcal{P}}_iC^\top$, $i\in\mathcal{N}$, such that  $A_i+L_iC = A_i(I-\widehat{\mathcal{P}}_iC^\top C)=A_i\widehat{\mathcal{P}}_i\bar{P}_i$, for all $i\in\mathcal{N}$. We denote 
 \begin{equation*}
    Q_{ij}\coloneqq \bar{S}_j-(A_i+L_iC)\bar{S}_i(A_i+L_iC)^\top, \quad i,j\in\mathcal{N},
\end{equation*}
and will now show that $Q_{ij}\succ 0$, $i,j\in\mathcal{N}$. It holds that $-\bar{P}_i\succcurlyeq -(\bar{P}_i+C^\top C)$, $i\in\mathcal{N}$, which, since 
\begin{equation*}
    Q_{ij} = \bar{S}_j-A_i(\bar{P}_i+C^\top C)^{-1}\bar{P}_i(\bar{P}_i+C^\top C)^{-1}A_i^\top,
\end{equation*}
for all $i,j\in\mathcal{N}$, implies that
            $Q_{ij}\succcurlyeq \bar{S}_j-A_i(\bar{P}_i+C^\top C)^{-1}A_i^\top\succ 0$ for all $i,j\in\mathcal{N}$, due to~\eqref{eq:twice-schur1}. Using Schur complement followed by a congruence transformation with $\operatorname{diag}\{\bar{P}_j,I\}$, $Q_{ij}\succ 0$ is equivalent to 
    \begin{equation*}
        \begin{bmatrix}\begin{smallmatrix}
            \bar{P}_j & \star\\
            (A_i+L_iC)^\top\bar{P}_j & \bar{P}_i
        \end{smallmatrix}\end{bmatrix}\succ 0,\text{ for all }i,j\in\mathcal{N}.
    \end{equation*}
    Multiplying by $\xi_i(\pi)\xi_j(\pi_+)$ and summing over all $i,j\in\mathcal{N}$ followed by applying Schur complement yields 
    \begin{equation*}
        P(\pi)-(A(\pi)+L(\pi)C)^\top P(\pi_+)(A(\pi)+L(\pi)C)\succ 0
    \end{equation*}
    for all $\pi_+,\pi\in\mathbb{P}$ with $P$ as in~\eqref{eq:poly-lyap}. Thus,~\eqref{eq:error-system} with $L$ in~\eqref{eq:poly-q-detectability-analysis-gain} is poly-QS and $V$~\eqref{eq:poly-lyap} is a poly-QLF. 
	
	\underline{Necessity for strictly polytopic systems:} Suppose the system~\eqref{eq:system} is poly-Q detectable.
    From~\eqref{eq:lyap-desc}, it follows that \[P(p_k)\succ (A(p_k)+L(k,\bm{p})C)^\top P(p_{k+1})(A(p_k)+L(k,\bm{p})C)\] for all $\bm{p}\in\mathcal{P}$ and $k\in\mathbb{N}$. Using the Schur complement, we find, for all $\bm{p}\in\mathcal{P}$ and $k\in\mathbb{N}$,
	\begin{equation}
		\begin{bmatrix}\begin{smallmatrix}
			P(p_k) & A^\top(p_k)\\
			\star & S(p_{k+1})
		\end{smallmatrix}\end{bmatrix} + \operatorname{He}\left(\begin{bmatrix}\begin{smallmatrix}
			C^\top\\
			0
		\end{smallmatrix}\end{bmatrix}L^\top(k,\bm{p})\begin{bmatrix}\begin{smallmatrix}
			0 & I
		\end{smallmatrix}\end{bmatrix}\right)\succ 0,
		\label{eq:interm1}
	\end{equation}
	where $S(\pi)=P^{-1}(\pi)\succ 0$, $\pi\in\mathbb{P}$, $k\in\mathbb{N}$ and $\bm{p}\in\mathcal{P}$. By pointwise application of Lemma~\ref{lem:proj-lem} for all $k\in\mathbb{N}$ and $\bm{p}\in\mathcal{P}$, there exists $L$ satisfying~\eqref{eq:interm1} if and only if
	\begin{equation}
		P(\pi)\succ 0\text{ and }\begin{bmatrix}\begin{smallmatrix}
			C & 0
		\end{smallmatrix}\end{bmatrix}_{\perp}^\top\begin{bmatrix}\begin{smallmatrix}
			P(\pi) & A^\top(\pi)\\
			\star & S(\pi_+)
		\end{smallmatrix}\end{bmatrix}\begin{bmatrix}\begin{smallmatrix}
			C & 0
		\end{smallmatrix}\end{bmatrix}_{\perp}\succ 0,
		\label{eq:Cperp}
	\end{equation}
	for all $\pi_+,\pi\in\mathbb{P}$. Using continuity of $P$, $S$ and $A$, we can apply Lemma~\ref{lem:finsler} to conclude that the second condition in~\eqref{eq:Cperp} holds if and only if there exists $\mu\in\mathbb{R}$ such that 
 \begin{equation*}
     \begin{bmatrix}\begin{smallmatrix}
        P(\pi) & A^\top(\pi)\\
        \star & S(\pi_+)
    \end{smallmatrix}\end{bmatrix} + \mu\begin{bmatrix}\begin{smallmatrix}
        C^\top\\
        0
    \end{smallmatrix}\end{bmatrix}\begin{bmatrix}\begin{smallmatrix}
        C & 0
    \end{smallmatrix}\end{bmatrix}\succ 0\text{ for all }\pi_+,\pi\in\mathbb{P},
 \end{equation*}
 which, using Schur complement, is equivalent to $P(\pi) - A^\top(\pi)P(\pi_+)A(\pi) + \mu C^\top C \succ 0$ for all $\pi_+,\pi\in\mathbb{P}$. It follows, due to strict polytopicity of~\eqref{eq:system}, that $\bar{P}_i\succ 0$ and $\bar{P}_i-A_i^\top \bar{P}_jA_i + \mu C^\top C\succ 0$ for all $i,j\in\mathcal{N}$. Since $C^\top C\succcurlyeq 0$,~\eqref{eq:poly-q-detectability-analysis} holds with $\bar{\mu}>\max\{0,\mu\}$ and $\bar{P}_i\leftarrow (1/\bar{\mu})\bar{P}_i$.
\end{proof}
\noindent The conditions above provide a non-conservative test for poly-Q detectability of strictly polytopic systems~\eqref{eq:system}, directly in line with the LTI test~\eqref{eq:lti-det-lmi}. Hence, we know, for any strictly polytopic system~\eqref{eq:system}, that, if the conditions in Theorem~\ref{thm:poly-q-detectability-analysis} are infeasible, then the system~\eqref{eq:system} is not poly-Q detectable. 
After obtaining a positive-definite solution to~\eqref{eq:poly-q-detectability-analysis}, we can construct an observer gain in~\eqref{eq:poly-q-detectability-analysis-gain}, which renders~\eqref{eq:error-system} poly-QS and whose implementation, importantly, only requires $p_{k}$ at time $k\in\mathbb{N}$.
\setcounter{thm}{\theremark}\stepcounter{remark}
\begin{rem}\label{rem:polytopic-L}
    Surprisingly, $L$ in~\eqref{eq:poly-q-detectability-analysis-gain} is polytopic and depends only on the current value of the parameters $p_k$ despite the fact that the considered detectability notion in Definition~\ref{dfn:poly-quadratic-detectability} allows $L$ to depend arbitrarily on the entire parameter sequence $\bm{p}$. Hence, a strictly polytopic system~\eqref{eq:system} is poly-Q detectable (without any prior restrictions on $L(k,\bm{p})$) if and only if it admits a polytopic $L(k,\bm{p})=\bar{L}(p_k)$. In other words, if the system~\eqref{eq:system} is poly-Q detectable using the \emph{complex} observer structure, then it is also poly-Q detectable using the \emph{simpler} observer structure, and, hence, the added complexity does not help us (and knowing other values of $p_l$ for $l\neq k$ is not helpful). Moreover, existence of an observer of either structure is equivalent to feasibility of the LMIs in Theorem~\ref{thm:poly-q-detectability-analysis}, as summarized in Fig.~\ref{fig:poly-q-det}. 
\end{rem}

\begin{figure}[!tb]
    \centering
    \includegraphics[scale=.16]{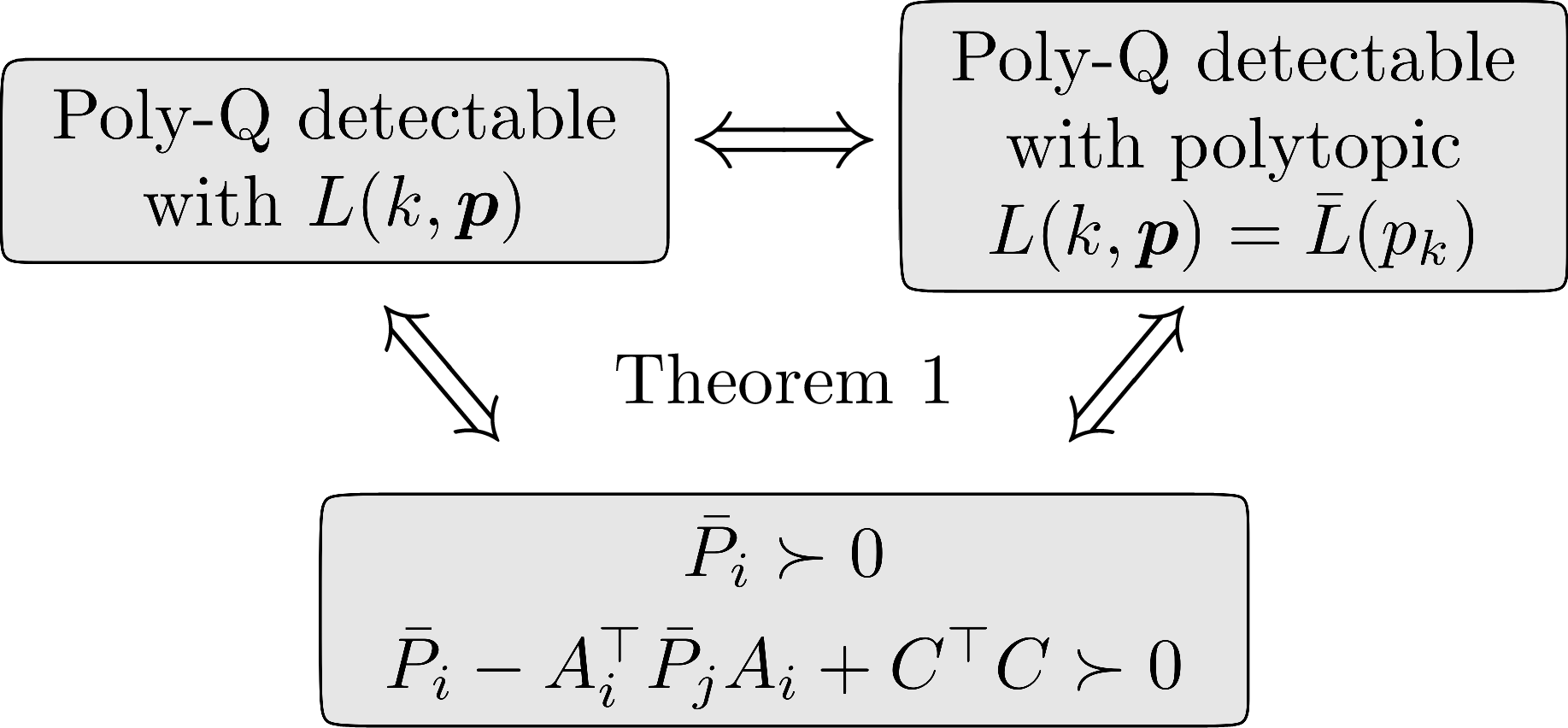}
    \caption{Overview of poly-Q detectability conditions.}
    \label{fig:poly-q-det}
\end{figure}

We conclude this section by showing that Theorem~\ref{thm:poly-q-detectability-analysis}, indeed, generalizes~\eqref{eq:lti-det-lmi}. We recover~\eqref{eq:lti-det-lmi} by substituting $P_i\leftarrow P$ and $A_i\leftarrow A$, with $P\in\mathbb{S}_{\succ 0}^{n_x}$ and $A\in\mathbb{R}^{n_x\times n_x}$, in~\eqref{eq:poly-q-detectability-analysis}. Making the same substitutions in~\eqref{eq:poly-q-detectability-analysis-gain} yields
\begin{equation*}
    L = -A(P+C^\top C)^{-1}C^\top = -ASC^\top(I+CSC^\top)^{-1},
\end{equation*}
where we used the Woodbury identity to obtain the latter expression, in which $P = S^{-1}$. If $P$ satisfies~\eqref{eq:lti-det-lmi}, this is precisely the observer gain obtained from~\eqref{eq:lti-det-lmi} in the LTI setting, see, e.g.,~\cite{Duan2013}. Thus, Theorem~\ref{thm:poly-q-detectability-analysis} truly generalizes the well-known Lyapunov-based stabilizability results from the LTI setting to polytopic LPV systems~\eqref{eq:system}.

\subsection{Poly-Q stabilizability test}\label{sec:stabilizability}
Inspired by the generalization of~\eqref{eq:lti-det-lmi} in Theorem~\ref{thm:poly-q-detectability-analysis}, we aim to generalize the stabilizability test in~\eqref{eq:lti-stab-lmi}. First, we state non-conservative parameter-dependent conditions.
\setcounter{thm}{\thetheorem}\stepcounter{theorem}
\begin{thm}\label{thm:poly-quadratic-stabilizability-nec-suff}
    The polytopic system~\eqref{eq:system} is poly-Q stabilizable, if and only if there exist $\bar{S}_i\in\mathbb{S}^{n_x}_{\succ 0}$ and $\bar{P}_i=\bar{S}_i^{-1}$, $i\in\mathcal{N}$, such that $S(\pi)=P^{-1}(\pi)$, with $P$ in~\eqref{eq:poly-lyap}, satisfies
    \begin{equation}
        S(\pi_+)-A(\pi)S(\pi)A^\top(\pi) + BB^\top\succ 0,
        \label{eq:poly-quadratic-stabilizability-nec-suff}
    \end{equation}
    for all $\pi_+,\pi\in\mathbb{P}$. Moreover, in that case, 
	\begin{equation}
		K(k,\bm{p}) = -B^\top(S(p_{k+1})+BB^\top)^{-1}A(p_k),
		\label{eq:poly-q-stabilizability-analysis-gain}
	\end{equation}
    for $k\in\mathbb{N}$ and $\bm{p}\in\mathcal{P}$, renders the closed-loop system~\eqref{eq:closed-loop} poly-QS with poly-QLF~\eqref{eq:poly-lyap}.
\end{thm}
\vspace*{-\belowdisplayskip}
\begin{proof}
    \underline{Sufficiency:} Suppose there exist $\bar{S}_i\in\mathbb{S}^{n_x}_{\succ 0}$ and $\bar{P}_i=\bar{S}_i^{-1}$, $i\in\mathcal{N}$, such that $S(\pi)=P^{-1}(\pi)$ with $P$ in~\eqref{eq:poly-lyap} satisfies~\eqref{eq:poly-quadratic-stabilizability-nec-suff} for all $\pi_+,\pi\in\mathbb{P}$. Let $\widehat{\mathcal{S}}(\pi) \coloneqq (S(\pi)+BB^\top)^{-1}$, $\pi\in\mathbb{P}$. Using Schur complement twice yields
 \begin{equation}\label{eq:twice-schur2}
    P(\pi) - A^\top(\pi)\widehat{\mathcal{S}}(\pi_+)A(\pi)\succ 0,
 \end{equation}
for all $\pi_+,\pi\in\mathbb{P}$. It follows, from $\bar{S}_i\succ 0$, $i\in\mathcal{N}$, that $S(\pi)\succ 0$ for all $\pi\in\mathbb{P}$. Let $K$ be as in~\eqref{eq:poly-q-stabilizability-analysis-gain} such that 
\begin{align*}
    &A_{\mathrm{cl}}(k,\bm{p})\coloneqq A(p_k)+BK(k,\bm{p})=\\
    &\quad(I-BB^\top\widehat{\mathcal{S}}(p_{k+1}))A(p_k)=S(p_{k+1})\widehat{\mathcal{S}}(p_{k+1})A(p_k),
\end{align*}
for all $k\in\mathbb{N}$, $\bm{p}\in\mathcal{P}$. We denote, for $k\in\mathbb{N}$, $\bm{p}\in\mathcal{P}$, \begin{align*}
    &Q(k,\bm{p}) \coloneqq P(p_k)-A_{\mathrm{cl}}^\top(k,\bm{p}) P(p_{k+1})A_{\mathrm{cl}}(k,\bm{p}).
\end{align*}
Next, we show that $Q(k,\bm{p})\succ 0$ for all $k\in\mathbb{N}$ and $\bm{p}\in\mathcal{P}$. It holds that $S(\pi)\preccurlyeq-\widehat{\mathcal{S}}^{-1}(\pi)$ for all $\pi\in\mathbb{P}$, which, using \[Q(k,\bm{p})=P(p_k)-A^\top(p_k)\widehat{\mathcal{S}}(p_{k+1}) S(p_{k+1})\widehat{\mathcal{S}}(p_{k+1})A(p_k),\] for $k\in\mathbb{N}$, $\bm{p}\in\mathcal{P}$, implies that, for all $k\in\mathbb{N}$ and $\bm{p}\in\mathcal{P}$,
\begin{align*}
    &Q(k,\bm{p})\succcurlyeq P(p_k)- A^\top(p_k)\widehat{\mathcal{S}}(p_{k+1})A(p_k)\succ 0,
\end{align*}
due to~\eqref{eq:twice-schur2}.
It follows, from $Q(k,\bm{p})\succ 0$ and $P$ being uniformly bounded on $\mathbb{P}$, that $V$ in~\eqref{eq:poly-lyap} is a poly-QLF for~\eqref{eq:closed-loop} and, thus, ~\eqref{eq:closed-loop} with $K$ in~\eqref{eq:poly-q-stabilizability-analysis-gain} is poly-QS.

    \underline{Necessity:} Suppose that system~\eqref{eq:system} is poly-Q stabilizable.
    Then, $\bar{P}_i\succ 0$, $\bar{S}_i=\bar{P}_i^{-1}$, $i\in\mathcal{N}$, and $S(\pi)=P^{-1}(\pi)\succ 0$, $\pi\in\mathbb{P}$. By Schur complement,~\eqref{eq:lyap-desc} yields, for all $\bm{p}\in\mathcal{P}$, $k\in\mathbb{N}$,
	\begin{equation}
		\begin{bmatrix}\begin{smallmatrix}
			P(p_k) & \star\\
			A(p_k) & S(p_{k+1}) 
		\end{smallmatrix}\end{bmatrix} + \operatorname{He}\left(\begin{bmatrix}\begin{smallmatrix}
			0\\
			B
		\end{smallmatrix}\end{bmatrix}K(k,\bm{p})\begin{bmatrix}\begin{smallmatrix}
			I & 0
		\end{smallmatrix}\end{bmatrix}\right)\succ 0.
		\label{eq:interm3}
	\end{equation}
	By pointwise application of Lemma~\ref{lem:proj-lem} for all $k\in\mathbb{N}$ and $\bm{p}\in\mathcal{P}$, there exists $K$ satisfying~\eqref{eq:interm3} if and only if
	\begin{equation*}
		S(\pi)\succ 0\text{ and }\begin{bmatrix}\begin{smallmatrix}
			0 & B^\top
		\end{smallmatrix}\end{bmatrix}_{\perp}^\top\begin{bmatrix}\begin{smallmatrix}
			P(\pi) & \star\\
			A(\pi) & S(\pi_+) 
		\end{smallmatrix}\end{bmatrix}\begin{bmatrix}\begin{smallmatrix}
			0 & B^\top
		\end{smallmatrix}\end{bmatrix}_{\perp}\succ 0,
	\end{equation*}
	for all $\pi_+,\pi\in\mathbb{P}$. We can complete the proof by applying Lemma~\ref{lem:finsler} pointwise for all $\pi_+,\pi\in\mathbb{P}$ and following the same steps as at the end of the proof of Theorem~\ref{thm:poly-q-detectability-analysis}.
\end{proof}
\noindent Despite starting from the general stabilizability notion in Definition~\ref{dfn:poly-q-stabilizability}, the controller gain in~\eqref{eq:poly-q-stabilizability-analysis-gain} only depends on $p_k$ and $p_{k+1}$ and features a specific structure. As discussed in Remark~\ref{rem:polytopic-L} for the observer design problem, this means that if the system~\eqref{eq:system} is poly-Q stabilizable by a \emph{complex} controller, then it is also poly-Q stabilizable using the \emph{simpler} controller structure in~\eqref{eq:poly-q-stabilizability-analysis-gain} (and, hence, knowing other values of $p_l$ for $l\not\in\{k,k+1\}$ is not helpful). Implementation of~\eqref{eq:poly-q-stabilizability-analysis-gain} only requires $p_k$ and $p_{k+1}$ at time $k\in\mathbb{N}$. If $p_{k+1}$ is not known ahead of time, an estimate $\hat{p}_{k+1}$ can be used and Theorem~\ref{thm:poly-quadratic-stabilizability-nec-suff} naturally leads to input-to-state stability with respect to $p_{k+1}-\hat{p}_{k+1}$.

To link Theorem~\ref{thm:poly-quadratic-stabilizability-nec-suff} to the LTI result~\eqref{eq:lti-stab-lmi}, we substitute $A(\pi)\leftarrow A$ and $S(\pi)\leftarrow S$, with $A\in\mathbb{R}^{n_x\times n_x}$ and $S\in\mathbb{S}^{n_x}_{\succ 0}$, in~\eqref{eq:poly-q-stabilizability-analysis-gain} to recover (using the Woodbury identity~\cite{Golub1996})
\begin{equation*}
	K = -B^\top(S+BB^\top)^{-1}\bar{A} = -(I+B^\top PB)^{-1}B^\top P\bar{A},
\end{equation*}
with $S=P^{-1}$ satisfying~\eqref{eq:lti-stab-lmi}, which is precisely the stabilizing controller gain obtained from~\eqref{eq:lti-stab-lmi} in the LTI setting, see, e.g.,~\cite{Duan2013}. Similarly, the conditions in Theorem~\ref{thm:poly-quadratic-stabilizability-nec-suff} reduce to~\eqref{eq:lti-stab-lmi} when making the same substitutions and, thus, generalize~\eqref{eq:lti-stab-lmi} to the polytopic LPV setting. To systematically solve the conditions in Theorem~\ref{thm:poly-quadratic-stabilizability-nec-suff}, we typically translate them to conditions in terms of the vertices. The following corollary follows by evaluating~\eqref{eq:poly-quadratic-stabilizability-nec-suff} at the vertices and, hence, its proof is omitted.
\setcounter{thm}{\thecorollary}\stepcounter{corollary}
\begin{cor}\label{cor:nec-cond-stab}
    The strictly polytopic system~\eqref{eq:system} is poly-Q stabilizable only if there exist $\bar{S}_i\in\mathbb{S}^{n_x}_{\succ 0}$, $i\in\mathcal{N}$, such that
    \begin{equation}
    \bar{S}_j-A_i\bar{S}_iA_i^\top + BB^\top \succ 0,\text{ for all }i,j\in\mathcal{N}.
    \label{eq:poly-q-stabilizability-analysis-nec}
    \end{equation}
\end{cor}
\vspace*{-\belowdisplayskip}
\noindent It follows, for any strictly polytopic system~\eqref{eq:system}, that, if the conditions in Corollary~\ref{cor:nec-cond-stab} are infeasible, then the system~\eqref{eq:system} is not poly-Q stabilizable. Unfortunately, there is no proof--to the best of our knowledge--that ~\eqref{eq:poly-q-stabilizability-analysis-nec} is\emph{ sufficient}. The challenge is that it is extremely difficult to derive, from~\eqref{eq:poly-q-stabilizability-analysis-nec}, a condition that is both linear in $\bar{P}_i=\bar{S}_i^{-1}$ and that does not contain any products between matrices which depend on the same index $i$ or $j$, which is crucial in proving sufficiency of~\eqref{eq:poly-q-stabilizability-analysis-nec}. By introducing additional slack variables in Theorem~\ref{thm:poly-quadratic-stabilizability-nec-suff}, however, we obtain the conditions below, which are \emph{sufficient}.
\setcounter{thm}{\thetheorem}\stepcounter{theorem}
\begin{thm}\label{thm:poly-q-stabilizability-analysis}
	The polytopic system~\eqref{eq:system} is poly-Q stabilizable, if there exist matrices $\bar{S}_i\in\mathbb{S}^{n_x}$ and $X_{i}\in\mathbb{R}^{n_x\times n_x}$, $\bar{P}_i=\bar{S}_i^{-1}$, $i\in\mathcal{N}$, such that
	\begin{equation}
		\begin{bmatrix}\begin{smallmatrix}
			X_i+X_i^\top -A_i\bar{S}_iA_i^\top + BB^\top & \star\\
			X_i & \bar{S}_j
		\end{smallmatrix}\end{bmatrix}\succ 0,\text{ for all }i,j\in\mathcal{N}.
		\label{eq:poly-q-stabilizability-analysis}
	\end{equation}
	Moreover, in that case, $K$ in~\eqref{eq:poly-q-stabilizability-analysis-gain} with $P$ in~\eqref{eq:poly-lyap} and $S(\pi)=P^{-1}(\pi)$ renders the closed-loop system~\eqref{eq:closed-loop} poly-QS with poly-QLF~\eqref{eq:poly-lyap}.
\end{thm}
\vspace*{-\belowdisplayskip}
\begin{proof}
    Suppose that there exist $\bar{P}_i\in\mathbb{S}^{n_x}$ and $X_i\in\mathbb{R}^{n_x\times n_x}$, $i\in\mathcal{N}$, such that~\eqref{eq:poly-q-stabilizability-analysis} holds with $\bar{S}_i\coloneqq \bar{P}^{-1}_i$, $i\in\mathcal{N}$. It follows that $\bar{S}_i\succ 0$, $i\in\mathcal{N}$, and, by congruence transformation with $\operatorname{diag}\{I,\bar{P}_j\}$ and Schur complement, it follows from~\eqref{eq:poly-q-stabilizability-analysis} that
    \begin{equation}
        \begin{bmatrix}\begin{smallmatrix}
			\bar{P}_i & \star & \star\\
			A_i & \operatorname{He}(X_i) +BB^\top & \star\\
			0 & \bar{P}_jX_i & \bar{P}_j
		\end{smallmatrix}\end{bmatrix}\succ 0,\text{ for all }i,j\in\mathcal{N}.
        \label{eq:interm9}
    \end{equation}
    Multiplying~\eqref{eq:interm9} by $\xi_i(\pi)\xi_j(\pi_+)$ and summing over all $i,j\in\mathcal{N}$, we obtain, for all $\pi_+,\pi\in\mathbb{P}$,
	\begin{equation}
		\begin{bmatrix}\begin{smallmatrix}
			P(\pi) & \star & \star\\
			A(\pi) & \operatorname{He}(X(\pi)) +BB^\top & \star\\
			0 & P(\pi_+)X(\pi) & P(\pi_+)
		\end{smallmatrix}\end{bmatrix}\succ 0,\label{eq:interm4}
	\end{equation}
	with $P(\pi)=\sum_{i\in\mathcal{N}}\xi_i(\pi)\bar{P}_i$, as in~\eqref{eq:poly-lyap}, and $X(\pi)=\sum_{i\in\mathcal{N}}\xi_i(\pi)X_i$, $\pi\in\mathbb{P}$. Since $\bar{P}_i\succ 0$, $i\in\mathcal{N}$, we have $P(\pi)\succ 0$, $\pi\in\mathbb{P}$. From~\eqref{eq:interm4}, we obtain $X(\pi)+X^\top(\pi)-X^\top(\pi)P(\pi_+)X(\pi)- A(\pi)S(\pi)A^\top(\pi)+BB^\top\succ 0$,
for all $\pi_+,\pi\in\mathbb{P}$. By pointwise application of Lemma~\ref{lem:youngs}, we conclude that $S$ satisfies~\eqref{eq:poly-quadratic-stabilizability-nec-suff} for all $\pi_+,\pi\in\mathbb{P}$. Thus, by Theorem~\ref{thm:poly-quadratic-stabilizability-nec-suff}, $K$ as in~\eqref{eq:poly-q-stabilizability-analysis-gain} with $S(\pi)=P^{-1}(\pi)$, $\pi\in\mathbb{P}$, and $P$ as in~\eqref{eq:poly-lyap} renders~\eqref{eq:closed-loop} poly-QS with poly-QLF~\eqref{eq:poly-lyap}.
\end{proof}
\noindent Theorem~\ref{thm:poly-q-stabilizability-analysis} provides LMI-based conditions to guarantee that a given polytopic system~\eqref{eq:system} is poly-Q stabilizable. Observe that, in contrast with~\eqref{eq:poly-q-detectability-analysis-gain} in the context of poly-Q detectability, $K$ in~\eqref{eq:poly-q-stabilizability-analysis-gain} is not polytopic and depends on both $p_k$ and $p_{k+1}$. This emphasizes that the obtained polytopicity and $p_k$-dependence of $L$ in~\eqref{eq:poly-q-detectability-analysis-gain}, as observed in Remark~\ref{rem:polytopic-L}, are not at all natural and/or obvious. Fig.~\ref{fig:poly-q-stab} summarizes all results on poly-Q stabilizability presented in this technical communique and shows how the different conditions are related.

\begin{figure}[!t]
    \centering
    \includegraphics[scale=.16]{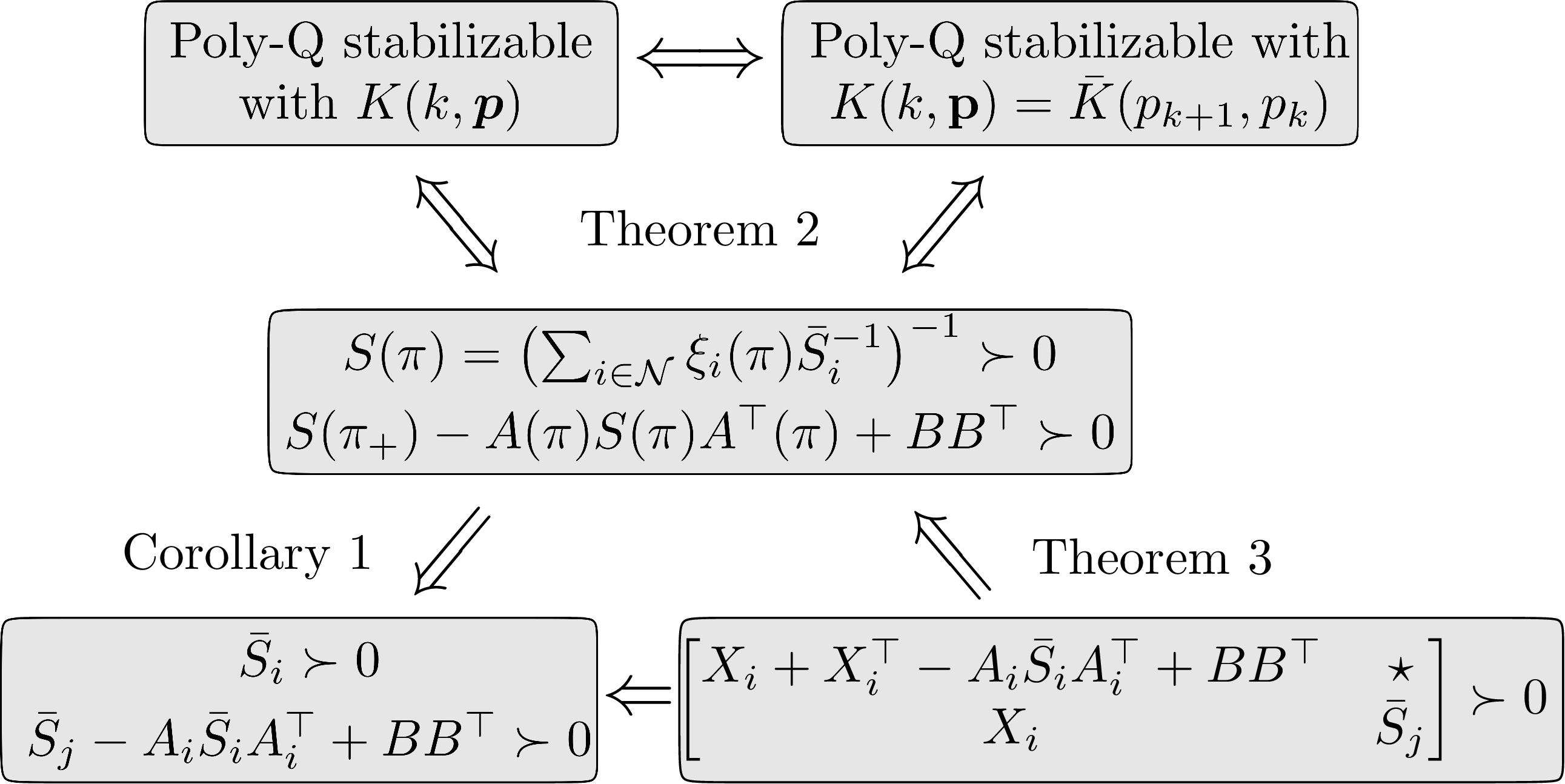}
    \caption{Overview of poly-Q stabilizability conditions.}
    \label{fig:poly-q-stab}
\end{figure}

\setcounter{thm}{\theremark}\stepcounter{remark}
\begin{rem}\label{rem:duality}
    We do not obtain duality between stabilizability and detectability as in the LTI case also for poly-Q stabilizability and poly-Q detectability. This manifests through \begin{enumerate*}[label=(\alph*)] \item the fact that $L$ in~\eqref{eq:poly-q-detectability-analysis-gain} depends only on $p_k$ whereas $K$ in~\eqref{eq:poly-q-stabilizability-analysis-gain} depends also on $p_{k+1}$, and \item all conditions in Fig.~\ref{fig:poly-q-det} being equivalent while Corollary~\ref{cor:nec-cond-stab} is not equivalent to poly-Q stabilizability (see Fig.~\ref{fig:poly-q-stab}). \end{enumerate*} Interestingly, duality is recovered when considering common quadratic Lyapunov functions, i.e., $\bar{P}_i=\bar{P}_j$ for all $i,j\in\mathcal{N}$. Then, both $L$ in~\eqref{eq:poly-q-detectability-analysis-gain} and $K$ in~\eqref{eq:poly-q-stabilizability-analysis-gain} depend only on $p_k$ and all conditions in Fig.~\ref{fig:poly-q-stab} turn out to be equivalent using Lemma~\ref{lem:proj-lem} in the Appendix. In fact, in that case~\eqref{eq:poly-q-stabilizability-analysis-nec} is~\eqref{eq:poly-q-detectability-analysis} with $A_i\leftarrow A_i^\top$ and $C\leftarrow B^\top$.
\end{rem}

\section{Conclusions}\label{sec:conclusions}
In this technical communique, we extended celebrated detectability and stabilizability tests for LTI systems to a class of discrete-time polytopic LPV systems.

\bibliographystyle{plain}        
\bibliography{refs}           

\section*{Appendix}
\appendix

\section{Existing technical lemmas}
\setcounter{thm}{\thelemma}\stepcounter{lemma}
\begin{lem}[{\cite{Gahinet1994}}]\label{lem:proj-lem}
	Let $\Psi\in\mathbb{S}^{n}$, $\Gamma\in\mathbb{R}^{m\times n}$ and $\Omega\in\mathbb{R}^{p\times n}$. There exists $\Lambda\in\mathbb{R}^{m\times p}$ such that $\Psi + \Gamma^\top\Lambda\Omega + \Omega^\top\Lambda^\top\Gamma\succ 0$, if and only if $\Gamma_{\perp}^\top\Psi\Gamma_{\perp}\succ 0\text{ and }\Omega_\perp^\top\Psi\Omega_{\perp}\succ 0$.
\end{lem}

\setcounter{thm}{\thelemma}\stepcounter{lemma}
\begin{lem}[{\cite{Ishihara2017}}]\label{lem:finsler}
	Let $\Psi\colon\mathbb{P}\rightarrow\mathbb{S}^{n}$ and $\Gamma\colon\mathbb{P}\rightarrow\mathbb{R}^{m\times n}$, with $\mathbb{P}\subset\mathbb{R}^{p}$ being compact, be continuous. It holds that $\Gamma_{\perp}^\top(\pi)\Psi(\pi)\Gamma_{\perp}(\pi)\succ 0$ for all $\pi\in\mathbb{P}$, if and only if $\Psi(\pi) +\mu \Gamma^\top(\pi)\Gamma(\pi)\succ 0$ for all $\pi\in\mathbb{P}$, for some $\mu\in\mathbb{R}$.
\end{lem}

\setcounter{thm}{\thelemma}\stepcounter{lemma}
\begin{lem}[{\cite{Zhou1988}}]\label{lem:youngs}
	Let $X,Y\in\mathbb{R}^{m\times n}$ and $S\in\mathbb{S}^{m}_{\succ 0}$. It holds that $X^\top S^{-1}X + Y^\top SY\succcurlyeq X^\top Y+Y^\top X$.
\end{lem}

\end{document}